\documentclass[11pt]{amsart}
\usepackage[colorlinks,citecolor=red,urlcolor=blue,bookmarks=false,hypertexnames=true]{hyperref} 
\usepackage{amsfonts}
\usepackage{amssymb}
\usepackage{nicefrac,eufrak}
\usepackage{bbm}
\usepackage{enumerate}
\usepackage{url,bm}
\usepackage{color}      
\usepackage{capt-of}
\usepackage{graphicx}
\newcommand{\kJ}{\mathfrak{J}}

\newcommand{\R}{\mathbb{R}}
\newcommand{\N}{\mathbb{N}}

\newcommand{\Z}{\mathbb{Z}}
\newcommand{\F}{\mathcal{F}}
\newcommand{\A}{\mathcal{A}}

\newcommand{\T}{\mathcal{T}}
\newcommand{\B}{\mathcal{B}}
\newcommand{\calH}{\mathcal{H}}
\newcommand{\calP}{\mathcal{P}}
\newcommand{\calC}{\mathcal{C}}

\newcommand{\calN}{\mathcal{N}}
\newcommand{\X}{\mathcal{X}}

\newcommand{\w}{\omega}
\newcommand{\de}{\delta}

\newcommand{\g}{\gamma}
\newcommand{\W}{\Omega}
\newcommand{\htau}{\hat{\tau}}

\newcommand{\wh}[1]{\widehat{#1}}

\newtheorem{theorem}{Theorem}[section]
\newtheorem{lemma}[theorem]{Lemma}

\newtheorem{proposition}[theorem]{Proposition}

\newtheorem*{general-problem}{GP }
\newtheorem*{Problema A}{Problem A}
\newtheorem*{Problema B}{Problem B}

\newtheorem*{Problema C}{Problem C}
\theoremstyle{remark}
\newtheorem{remark}[theorem]{Remark}
\newtheorem{example}[theorem]{Example}

\theoremstyle{definition}
\newtheorem{definition}[theorem]{Definition}

\date{December 2019}
\subjclass[2010]{Primary 47A15, 94A12 ; Secondary 42C15, 42C40, 43A70.}
\keywords{LCA groups, Zak transform, invariant spaces, extra invariance, multiplicatively invariant spaces}
\thanks{The research of the authors is partially supported by grants: UBACyT 20020170100430BA, PICT 2014-1480 (ANPCyT) and CONICET PIP 11220150100355. In particular VP is also supported  by UBACyT 20020170200057BA and PICT-2016- 2616 (Joven).}

\begin{document}
\title[Extra-invariance of group actions. ]{Extra-invariance of group actions}

\author[ C. Cabrelli, C. A. Mosquera and V. Paternostro]{C. Cabrelli, C. A. Mosquera and V. Paternostro}

\address{\textrm{(C. Cabrelli)}
Departamento de Matem\'atica,
Facultad de Ciencias Exac\-tas y Naturales,
Universidad de Buenos Aires, Ciudad Universitaria, Pabell\'on I,
1428 Buenos Aires, Argentina and
IMAS-CONICET, Consejo Nacional de Investigaciones
Cient\'ificas y T\'ecnicas, Argentina}
\email{cabrelli@dm.uba.ar}

\address{\textrm{(C. A. Mosquera)}
Departamento de Matem\'atica,
Facultad de Ciencias Exac\-tas y Naturales,
Universidad de Buenos Aires, Ciudad Universitaria, Pabell\'on I,
1428 Buenos Aires, Argentina and
IMAS-CONICET, Consejo Nacional de Investigaciones
Cient\'ificas y T\'ecnicas, Argentina}
\email{mosquera@dm.uba.ar}

\address{\textrm{(V. Paternostro)}
Departamento de Matem\'atica,
Facultad de Ciencias Exac\-tas y Naturales,
Universidad de Buenos Aires, Ciudad Universitaria, Pabell\'on I,
1428 Buenos Aires, Argentina and
IMAS-CONICET, Consejo Nacional de Investigaciones
Cient\'ificas y T\'ecnicas, Argentina}
\email{vpater@dm.uba.ar}

 \begin{abstract}
Given discrete groups $\Gamma \subset \Delta$ we characterize  $(\Gamma,\sigma)$-invariant spaces that are also invariant under $\Delta$. This will be done in terms of subspaces that we  define using an appropriate  Zak transform and a particular partition of the underlying group. On the way, we obtain a new characterization of principal 
$(\Gamma,\sigma)$-invariant spaces in terms of the Zak transform of its generator. This result is in the spirit of the analogous in  the context of shift-invariant spaces in terms of the Fourier transform, which is very well-known. As a consequence of our results, we give a solution for the problem of finding the  $(\Gamma,\sigma)$-invariant space  nearest - in the sense of least squares - to a given set of data. 
\end{abstract}

\maketitle
\section{Introduction}
Given a measure space $(\X,\mu)$, we study spaces in $L^2(\X)$ that are invariant under the action $\sigma$ of a discrete group $\Gamma$ on $\X$. These spaces are called $(\Gamma,\sigma)$-invariant spaces and their structure was completely characterized in terms of appropriate range functions in \cite{BHP15}. They provide a unified framework for studying spaces invariant, for instance, under dilations, translations and the action of the shear matrix which are operators having a very featured role in Harmonic Analysis, wavelet and shearlet theory and signal processing. 

Since $\X$ might  not have group structure, we face with the lack of the Fourier transform that is the usual tool when studying spaces invariant under translations (see \cite{Bow00, CP10, BR14} and the references therein). To overcome this situation, the approach  is to work with a Zak transform adapted to this setting as was first done in \cite{HSWW10, BHP14} and then in \cite{BHP15}. 

In this paper we make use of the Zak transform to characterized $(\Gamma,\sigma)$-invariant spaces that are additionally invariant by another group $\Delta$ that also acts on $\X$ and whose action extends that of $\Gamma$. Under this situation, we say that the  
$(\Gamma,\sigma)$-invariant space has extra invariance  $\Delta$. The problem of characterizing extra invariance of an invariant space  was first proposed and solved in \cite{ACHKM10} where the authors consider shift-invariant spaces of $L^2(\R)$ (that is, spaces invariant  under integer translations) that are invariant by $\frac1{n}\Z$ for some $n\in\N_{>1}$. Then, the extra invariance was studied for shift-invariant spaces in several variables in \cite{ACP11} and in the context of locally compact abelian groups in \cite{ACP10}. In \cite{CMP17} the problem  was analyzed for multiplicatively invariant spaces.

Our   main result  gives necessary and sufficient conditions for a $(\Gamma,\sigma)$-invariant space to be invariant under another group $\Delta$. The conditions are stated in terms of some closed subspaces of $L^2(\X)$ that we defined using the Zak transform and a proper partition of the dual group of $\T$, which is  group containing $\Gamma$ and $\Delta$, i.e. $\Gamma\subseteq \Delta\subseteq \T$. The characterization that we obtain is inspired in the one given in \cite{ACP10, ACP11} but, due to the more complicated context, it requires to deal with some more technical issues. In particular, we needed to prove a characterization of principal $(\Gamma,\sigma)$-invariant space, which are those generated by a single function of $L^2(\X)$. 
This result is new in this context and it is interested  by itself because gives a more comprehensive picture of $(\Gamma,\sigma)$-invariant spaces. 

As a application,  we solve the problem of finding the $(\Gamma,\sigma)$-invariant space with extra invariance that best approximates a given finite set of functions of $L^2(\X)$ in the sense of least-squares.

The paper is organized as follows. We first review the basics of  $(\Gamma,\sigma)$-invariant spaces and the Zak transform in Section \ref{sec: Gamma-spaces}. There, we also prove a characterization of principal $(\Gamma,\sigma)$-invariant spaces which is one of the new results of this paper (see Theorem \ref{thm:ppal-invariant}). In Section \ref{sec:extra-invariance} we present and prove our main result (Theorem \ref{thm:ppal}) where we give necessary and sufficient conditions for a 
 $(\Gamma,\sigma)$-invariant space to have extra invariance in a given ``bigger" group. To give deeper insight, we  describe the heuristic behind this result in Subsection \ref{sec:heuristic}. Finally, in Section \ref{sec:decomposable} 
we establish the connection between $(\Gamma,\sigma)$-invariant spaces with extra invariance and decomposable  multiplicatively invariant spaces. Once this connection is clear we solve the approximation problem.  

\subsection*{Notation:}
Given an LCA group $G$ written additively, we denote by $\wh G$ its Pontryagin  dual group. By $m_G$ we denote 
the Haar measure associated to $G$. Since the dual of the dual group is topologically isomorphic to the original group, for  $\xi\in\wh G$ and $x\in G$ we write $(x,\xi)$ to indicate the character $\xi$ applied to $x$ (i.e. $\xi(x)$) or the character  $x$ applied to $\xi$. For a subgroup  $H\subseteq G$, its annihilator is denoted by $H^*$ and is defined as 
$H^*=\{ \xi\in\wh G:\, (h,\xi)=1 \,\forall \,h\in H\}$. It is well known that if $H$ is closed, $H^*$ is a closed subgroup of $\wh G$. 
\section{Invariant spaces under the action of a discrete group}\label{sec: Gamma-spaces}
\subsection{$(\Gamma,\sigma)$-invariant spaces}
We are interested in invariant spaces with respect to unitary operators arising from the action of a discrete  LCA group $\Gamma$ on a $\sigma$-finite measure space $(\X, \mu)$. We shall work with actions that are  {\it quasi-$\Gamma$-invariant},
a notion  introduced in \cite{HSWW10} and then extended to the non abelian case in \cite{BHP14}. 
 
Fix $\Gamma$ a discrete countable LCA group. Let $(\X,\mu)$ be a $\sigma$-finite measure space and $\sigma:\Gamma\times\X\to\X$ a measurable action satisfying the following conditions:
\begin{enumerate}
 \item [(i)] for each $\g\in\Gamma$ the map $\sigma_{\g}:\X\to\X$ given by $\sigma_{\g}(x):=\sigma(\g,x)$ is $\mu$-
measurable;
\item [(ii)] $\sigma_{\g}(\sigma_{\g'}(x))=\sigma_{\g+\g'}(x)$, for all $\g,\g'\in\Gamma$ and for all $x\in\X$;
\item [(iii)] $\sigma_{e}(x)=x$ for all $x\in\X$, where $e$ is the identity of $\Gamma$.
\end{enumerate}
The action $\sigma$ is said to be {\it quasi-$\Gamma$-invariant} if there exists  a measurable function $J_{\sigma}:\Gamma\times\X\to\R^{+}$, called {\it Jacobian of $\sigma$}, such that
$d\mu(\sigma_{\g}(x))=J_{\sigma}(\g,x)d\mu(x).$
To each quasi-$\Gamma$-invariant action $\sigma$ we can associate a unitary representation  $\Pi$ of $\Gamma$ on $L^2(\X)$ given by
$\Pi_{\g}f(x)=J_{\sigma}(-\g,x)^{\frac1{2}}f(\sigma_{-\g}(x)).$ Note that $\Pi$ depends on $\sigma$ although it is not  explicit in the notation. 

\begin{definition}\label{def:invariant-spaces}
Given a quasi-$\Gamma$-invariant action $\sigma$, we say that a closed subspace $V$ of $L^2(\X)$ is  {\it $(\Gamma,\sigma)$-invariant} if
$$f\in V\Longrightarrow \Pi_{\g}f\in V, \,\, \textrm{ for any }\,\,\g\in\Gamma.$$
\end{definition}

When $L^2(\X)$ is separable, each $(\Gamma,\sigma)$-invariant space is of the form
$V=S_\Gamma(\A):=\overline{\mbox{span} }\{\Pi_{\g}\varphi:\g\in\Gamma, \varphi\in\A\}$ for some at most countable set $\A\subseteq L^2(\X)$ and we say that $V$ is generated by $\A$. When $\A$ is a finite set we say that $V$ is {\it finitely generated} and when  $\A=\{\varphi\}$ we write $S_\Gamma(\A)=S_\Gamma(\varphi)$ and we call it {\it principal } $\Gamma$-invariant space. 

\medskip

We shall say that  the quasi-$\Gamma$-invariant action $\sigma$ has the {\it tiling property} if there exists  a measurable subset $C_\Gamma\subseteq \X$ such that $\mu(\X\setminus\bigcup_{\g\in\Gamma}\sigma_{\g}(C_\Gamma))=0$ and $\mu(\sigma_{\g}(C_\Gamma)\cap\sigma_{\g'}(C_\Gamma))=0$ whenever $\g\neq\g'$. 
We call $C_\Gamma$ a {\it tilling set} for $\sigma$.

\subsection{The Zak transform}
Let  $\sigma$ be a quasi-$\Gamma$-invariant action having the tiling property, with tilling set  $C_{\Gamma}.$
Denote by  $\wh \Gamma$   the Pontryagin dual of $\Gamma$.
The space $L^2(\widehat{\Gamma}, L^2(C_\Gamma))$ consists of all measurable functions $\Phi:\widehat{\Gamma}\to  L^2(C_\Gamma)$ with finite norm, where the norm is given by 
$$\|\Phi\|:=\left(\int_{\widehat{\Gamma}}\|\Phi(\w)\|_{L^2(C_{\Gamma})}\,dm_{\wh \Gamma}(\w)\right)^{\frac1{2}}.$$

With this notation we have.
\begin{proposition}\cite[Proposition 3.3]{BHP15}\label{prop:Zak}
 The mapping $Z_\Gamma: L^2(\X) \longrightarrow  L^2(\widehat{\Gamma}, L^2(C_\Gamma))$ defined by
 $$Z_\Gamma[\psi](\alpha)(x):=\sum_{\g\in\Gamma}
 \Pi_{\g}\psi(x)(-\g, \alpha), \quad \alpha\in\widehat{\Gamma},\,\, x\in C_\Gamma$$
 is an isometric isomorphism and it satisfies
$Z_\Gamma[\Pi_{\g}\psi](\alpha)=(\g,\alpha)Z_\Gamma[\psi](\alpha)$.
\end{proposition}
The map $Z_\Gamma$ it is well known and it is  called   {\it Zak transform}.  In \cite{HSWW10}, was generalized to study 
cyclic subspaces asociated to representations of certain LCA groups. It was 
extended to the noncomutative setting in \cite{BHP14} and used to characterized  $(\Gamma,\sigma)$-invariant spaces in terms of range functions in \cite{BHP15}.

In this paper we will consider the Zak transform associated with different groups.
In what follows we describe the setting where we shall work in, and  establish the connection between the Zak transforms involved. 

Let $\T$ be a  discrete and countable  LCA group such that $\Gamma$ is a subgroup of $\T$  with  $\T/\Gamma$ being finite; that is, $\Gamma$ is a uniform lattice of $\T$. Additionally, suppose that exists a  quasi-$\T$-invariant action $\sigma^{\T}$of $\T$ on $\X$ that extends $\sigma$. By this we mean that 
$\sigma^\T\mid_{\Gamma\times\X}=\sigma$. 
In this situation we can derive the following two properties:

\begin{lemma}\label{lem:extension}
 Let $J_{\sigma}$ and $J_{\sigma^\T}$ be the Jacobians associated to $\sigma$ and $\sigma^\T$ and let $\Pi$ and  $\Pi^{\T}$ be the correspondent  unitary representations induced by $\sigma$ and $\sigma^\T$, respectively. Then we have:
 \begin{enumerate}
  \item [(i)] For every $\g\in\Gamma$, $J_{\sigma^\T}(\g,x)=J_{\sigma}(\g,x)$, for $\mu$-a.e. $x\in\X$.
  \item [(ii)] For every $\g\in\Gamma$,  $\Pi^{\T}_\g=\Pi_\g$. 
 \end{enumerate}
\end{lemma}

\begin{proof}
 Item {\it (i)} is a direct consequence of the definition of the Jacobian and the relationship between $\sigma$ and $\sigma^\T$ since, for every $\g\in\Gamma$
 $$J_{\sigma^\T}(\g,x)d\mu(x)=d\mu(\sigma^\T_{\g}(x))=d\mu(\sigma_{\g}(x))=J_{\sigma}(\g,x)d\mu(x).$$
 Item {\it (ii)} follows from item {\it (i)}.
\end{proof}

When the action $\sigma^\T$ has the tiling property a tiling set for $\sigma$ can be constructed in terms of $C_\T$, the tiling set for $\sigma^\T$. Note that in particular  the tiling property of $\sigma^\T$ forces $\sigma$ to have it.
\begin{lemma}
 Let $C_\T$ be a tiling set for $\sigma^\T$ and let $\F=\{a_0,a_1,\dots,a_s\}\subseteq \T$ be a section for the quotient $\T/\Gamma$. Then, 
 \begin{equation}\label{eq:tiling-set}
 C_\Gamma:=\bigcup_{j=0}^s\sigma^\T_{-a_j}(C_\T) 
 \end{equation}
 is a tiling set for $\sigma$.
\end{lemma}

\begin{proof}
 First, let $\g\neq\g'\in\Gamma$. Then, using the properties of the actions we have 
 \begin{align*}
  \sigma_\g(C_\Gamma)\cap\sigma_{\g'}(C_\Gamma)&=\sigma_\g(\bigcup_{j=0}^s\sigma^\T_{-a_j}(C_\T))\cap\sigma_{\g'}(\bigcup_{k=0}^s\sigma^\T_{-a_k}(C_\T))\\
  &=(\bigcup_{j=0}^s\sigma^\T_{-a_j+\g}(C_\T))\cap(\bigcup_{k=0}^s\sigma^\T_{-a_k+\g'}(C_\T))\\
   &=\bigcup_{j=0}^s\bigcup_{k=0}^s(\sigma^\T_{-a_j+\g}(C_\T))\cap(\sigma^\T_{-a_k+\g'}(C_\T)).\\
 \end{align*}
Since $\mu(\sigma^\T_{-a_j+\g}(C_\T))\cap(\sigma^\T_{-a_k+\g'}(C_\T))=0$ for all $j,k$ because $C_\T$ is a tiling set for $\sigma^\T$, we obtain that $\mu(\sigma_\g(C_\Gamma)\cap\sigma_{\g'}(C_\Gamma))=0$.

Finally, since $\bigcup_{\g\in\Gamma}\sigma_\g(C_\Gamma)= \bigcup_{\g\in\Gamma}\bigcup_{j=0}^s\sigma^\T_{\g-a_j}(C_\T)=\bigcup_{\tau\in\T}\sigma^\T_{\tau}(C_\T)$ the result follows. 
\end{proof}

In order to simplify the notation and as it is suggested by Lemma \ref{lem:extension}, we shall, from now on, omit the superindexes on $\sigma^\T$ and $\Pi^{\T}$ and simply write $\sigma$ and $\Pi$ for the actions and representations of $\Gamma$ and $\T$ on $\X$. 

Let $\Gamma^*$ be the annihilator of $\Gamma$. As it well known, $\Gamma^*$ turns out to be a subgroup of $\widehat{\T}$ such that  $\Gamma^*\approx\widehat{\T/\Gamma}$, where $\approx$ means ``isomorphic groups".  In particular, since $\T/\Gamma$ is finite it holds that $\T/\Gamma\approx\widehat{\T/\Gamma}$ and then  one has that $\#\Gamma^*=\#\T /\Gamma$.
Take $\Omega\subseteq \widehat{\T}$ a measurable section of $\widehat{\T}/\Gamma^*$ which existence is guaranteed  by \cite[Lemma 1.1.]{Mac52} and \cite[Theorem 1]{FG68}. Then, since $\widehat{\Gamma}\approx\widehat{\T}/\Gamma^*$ one can  identify $\widehat{\Gamma}$ with $\Omega$.  Under this identification, $L^2(\Omega, L^2(C_\Gamma))$ is isometricaly isomorphic to $L^2(\widehat{\Gamma}, L^2(C_\Gamma))$ and this allows us to have the mapping  $Z_\Gamma$ of Proposition \ref{prop:Zak} defined on $\Omega$ instead 
$\widehat{\Gamma}$. Moreover, $Z_\Gamma$ can be extended to the whole $\widehat{\T}$ in a $\Gamma^*$-periodic way since, for $f\in L^2(\X)$
\begin{equation}\label{eq:Zak-periodic}
Z_\Gamma[f](\w+\g^*)(x)=\sum_{\g\in\Gamma}
\Pi_{\g}f(x)(-\g, \w+\g^*)=\sum_{\g\in\Gamma}
 \Pi_{\g}f(x)(-\g, \w)=Z_\Gamma[f](\w)(x),
 \end{equation}
for every $\g^*\in\Gamma^*$, for $m_{\widehat{\T}}$-a.e. $\w\in\Omega$ and for $\mu$-a.e. $x\in C_\Gamma$. 

We can now consider   the mapping $Z_\T$ defined in Proposition \ref{prop:Zak} for the group $\T$. That is, $Z_\T:L^2(\X) \longrightarrow  L^2(\widehat{\T}, L^2(C_\T))$ is given by 
$$ Z_\T[f](\hat{\tau})(x)=\sum_{\tau\in\T}
\Pi_{\tau}f(x)(-\tau, \htau),$$
 where $f\in L^2(\X)$, $\htau\in\wh\T$ and $x\in C_\T$.
The isomorphisms $Z_\Gamma$ and $Z_\T$ are related as it is shown in the next
lemma.

\begin{lemma}\label{lem:Zak-comparation}
Let $Z_\Gamma:L^2(\X) \longrightarrow  L^2(\W, L^2(C_\Gamma))$ and 
$Z_\T:L^2(\X) \longrightarrow  L^2(\widehat{\T}, L^2(C_\T))$ be the isometric isomorphisms defined in Proposition \ref{prop:Zak} for the groups $\Gamma$ and $\T$ respectively. Then, the following relationship holds:
\begin{equation*}
Z_\T[f](\w)(x)=
\sum_{j=0}^sJ_\sigma(-a_j,x)^{\frac1{2}}Z_\Gamma[f](\w)(\sigma_{-a_j}(x))(-a_j,\w),
\end{equation*}
for every $f\in L^2(\X)$, for $m_{\widehat{\T}}$-a.e. $\w\in\Omega$ and for $\mu$-a.e. $x\in C_\T$.

Moreover, for every $\g^*\in\Gamma^*$,
\begin{equation}\label{eq:Zak-Upsilon}
Z_\T[f](\w+\g^*)(x)=
\sum_{j=0}^sJ_\sigma(-a_j,x)^{\frac1{2}}Z_\Gamma[f](\w)(\sigma_{-a_j}(x))(-a_j,\w+\g^*).
\end{equation}
\end{lemma}

\begin{proof}
Using that  $\F=\{a_0,\dots, a_s\}$ is a section for the quotient $\T/\Gamma$, we have
\begin{align*}
Z_\T[f](\w)(x)&=\sum_{\tau\in\T}
\Pi_{\tau}f(x)(-\tau, \w)
 =\sum_{j=0}^s\sum_{\g\in\Gamma}
\Pi_{a_j}\Pi_{\g}f (x)(-\g, \w)(-a_j, \w)\\
&=\sum_{j=0}^sJ_\sigma(-a_j,x)^{\frac1{2}} \left[\sum_{\g\in\Gamma}
\Pi_{\g}f(\sigma_{-a_j}(x))(-\g, \w)\right](-a_j, \w)\\
&=\sum_{j=0}^sJ_\sigma(-a_j,x)^{\frac1{2}} Z_\Gamma[f](\w)(\sigma_{-a_j}(x))(-a_j, \w).
\end{align*}
Using \eqref{eq:Zak-periodic}, equation \eqref{eq:Zak-Upsilon} is obtained in a similar way.
\end{proof}

\begin{remark}\label{rem:Zak-matrix} 
Note that from Lemma \ref{lem:Zak-comparation} the relationship between $Z_\Gamma$ and $Z_\T$ can be formulated in terms of matrices as follows: For $\w\in\W$ and $x\in C_\T$ define 
$$V(\w,x)=(Z_\Gamma[f](\w)(\sigma_{-a_0}(x)),\dots, Z_\Gamma[f](\w)(\sigma_{-a_s}(x)))^t$$
and 
$$U(\w,x)=(Z_\T[f](\w+\g_0^*)(x),\dots,Z_\T[f](\w+\g_s^*)(x))^t,$$
where $\Gamma^*=\{\g_0^*,\dots, \g_s^*\}$. Furthermore, 
let $D(\w,x)$ be the $(s+1)\times (s+1)$ diagonal matrix defined by
$$D(\w,x)=\textrm{diag}(J_\sigma(-a_0,x)^{\frac1{2}}(-a_0,\w),\dots,J_\sigma(-a_s,x)^{\frac1{2}}(-a_s,\w))$$
and let $F$ be the $(s+1)\times (s+1)$ matrix 
$$F=\begin{bmatrix}
(-a_0,\g^*_0)&\dots&(-a_s,\g^*_0)\\
 \vdots & \ddots & \vdots\\
 (-a_0,\g^*_s)&\dots&(-a_s,\g^*_s)
\end{bmatrix},
$$
which is the Discrete Fourier transform matrix associated to the group $\T/\Gamma$.
Then, by Lemma \ref{lem:Zak-comparation}, we have that 
$$F . D(\w,x) . V(\w,x)=U(\w,x),$$
and this shows that $U(\w,x)$ is -up to $D(\w,x)$- the Fourier transform of $V(\w,x)$.
\end{remark}

In what follows we present a result that characterizes functions belonging to a principal 
$(\Gamma,\sigma)$-invariant space. It is in the spirit of \cite[Theorem 2.14]{dBDVR94}, as well as \cite[Theorem 4.3]{ACP11} and \cite[Theorem 4.4]{ACP10} (see also \cite{dBDVRI94,RS95}).

\begin{theorem}\label{thm:ppal-invariant}
Let $\psi\in L^2(\X)$. Then, if $f\in S_\Gamma(\psi)$, there exists a measurable $\Gamma^*$-periodic function $h$ defined on $\wh\T$ such that 
$$Z_\T[f](\htau)=h(\htau)Z_\T[\psi](\htau),$$
 for $m_{\wh\T}$-a.e. $\htau\in\wh\T$.
Moreover, if $f\in L^2(\X)$ is such that $Z_\T[f](\htau)=h(\htau)Z_\T[\psi](\htau),$ $m_{\wh\T}$-a.e. $\htau\in\wh\T$ for some measurable  $\Gamma^*$-periodic function $h$ defined on $\wh\T$, then $f\in S_\Gamma(\psi)$.
\end{theorem}

For the proof we will need the following lemma. 

\begin{lemma}\label{lem:ppal-charcaterization}
Let $\psi\in L^2(\X)$. Then the following hold true:
\begin{enumerate}
\item [(i)] If $f\in S_\Gamma(\psi)$, then there exists a measurable function $h$ defined on $\Omega$ such that $Z_\Gamma[f](\w)=h(\w)Z_\Gamma[\psi](\w)$, for $m_{\wh\T}$-a.e. $\w\in\Omega$.
\item [(ii)] If $f\in L^2(\X)$ is such that $Z_\Gamma[f](\w)=h(\w)Z_\Gamma[\psi](\w)$ $m_{\wh\T}$-a.e. $\w\in\Omega$ for some measurable  function $h$ defined on $\Omega$ then $f\in S_\Gamma(\psi)$.
\end{enumerate}
\end{lemma}

\begin{proof}
$(i)$. We begin by proving that 
\begin{equation}\label{eq:pall-orthogonal}
S_\Gamma(\psi)^\perp=\{g\in L^2(\X):\, \langle Z_\Gamma[g](\w), Z_\Gamma[\psi](\w)\rangle_{L^2(C)}=0, \,\,m_{\wh\T}\textrm{-a.e. }\w\in\Omega\}. 
\end{equation}
Let $g\in L^2(\X)$ and define, for $\w\in\Omega$,  $F_g(\w):=\langle Z_\Gamma[g](\w), Z_\Gamma[\psi](\w)\rangle_{L^2(C)}$. Then, $F_g\in L^1(\Omega)\subseteq L^2(\Omega)$ and for all $\gamma\in\Gamma$ we have that 
\begin{align*}
  \widehat{F_g}(\gamma)=&\int_{\Omega}(-\g, \w)\langle Z_\Gamma[g](\w), Z_\Gamma[\psi](\w)\rangle_{L^2(C)}\,dm_{\widehat{\T}}(\w)\\
  =&\int_{\Omega}\langle Z_\Gamma[g](\w), (\g, \w)Z_\Gamma[\psi](\w)\rangle_{L^2(C)}\,dm_{\widehat{\T}}(\w)\\
   =&\langle Z_\Gamma[g], Z_\Gamma[\Pi_\g\psi]\rangle=\langle g, \Pi_\g\psi\rangle.
\end{align*}
Thus, noticing that  a function in $L^2(\Omega)$ is zero if and only if all its Fourier coefficients are zero and that $S_\Gamma(\psi)^\perp=\{g\in L^2(\X):\, \langle g, \Pi_\g\psi\rangle=0, \forall\,\,\g\in\Gamma\}$, \eqref{eq:pall-orthogonal} follows.  

Now, given $g\in L^2(\X)$, define  $h_g(\w):=\frac{\langle Z_\Gamma[g](\w), Z_\Gamma[\psi](\w)\rangle_{L^2(C)}}{\|Z_\Gamma[\psi](\w)\|^2_{L^2(C)}}{\bf 1 }_{E_\psi}(\w)$ for $m_{\wh\T}$-a.e. $\w\in\Omega$,     where $E_\psi:=\{\w\in\Omega:\, Z_\Gamma[\psi](\w)\neq0\}$. Then, $|h_g(\w)|^2\|Z_\Gamma[\psi](\w)\|^2_{L^2(C)}\leq \|Z_\Gamma[g](\w)\|^2_{L^2(C)}$ $m_{\wh\T}$-a.e. $\w\in\Omega$.

Let  $P$ be  the orthogonal projection of $L^2(\X)$ onto $S_\Gamma(\psi)$, then 
\begin{equation}\label{eq:proj-hg}
Z_\Gamma[Pg](\w)=h_g(\w)Z_\Gamma[\psi](\w)
\end{equation}
 for $m_{\wh\T}$-a.e. $\w\in\Omega$.
 Indeed, 
since $Z_\Gamma$ is an isometry, it is easily seen that $Z_\Gamma[Pg]=\calP Z_\Gamma[g]$ where $\calP$ denotes the orthogonal projection of $L^2(\Omega, L^2(C))$ onto $Z_\Gamma[S_\Gamma(\psi)]$. 
Define now $Q: L^2(\Omega, L^2(C))\to L^2(\Omega, L^2(C))$ by $Q(Z_\Gamma[g])(\w)=h_g(\w)Z_\Gamma[\psi](\w)$ for $m_{\wh\T}$-a.e. $\w\in\Omega$. Then,  $Q$ is well-defined and continuous since 
$$\|Q(Z_\Gamma[g])\|^2 = \int _\Omega|h_g(\w)|^2\|Z_\Gamma[\psi](\w)\|^2_{L^2(C)}\,dm_{\widehat{\T}}(\w)\leq \|Z_\Gamma[g]\|^2.$$
Since $Z_\Gamma[S_\Gamma(\psi)^\perp]=(Z_\Gamma[S_\Gamma(\psi)])^\perp$ and \eqref{eq:pall-orthogonal} holds, we have that $Q(Z_\Gamma[g])=0$ if $g\in S_\Gamma(\psi)$ and thus $Q=\calP$ on $(Z_\Gamma[S_\Gamma(\psi)])^\perp$. To see that $Q$ and $\calP$ also agree on $Z_\Gamma[S_\Gamma(\psi)]$, take $\g\in\Gamma$ and note that for $g=\Pi_\g\psi$, $h_g(\w)=(\g,\w){\bf 1 }_{E_\psi}(\w)$ for $m_{\wh\T}$-a.e. $\w\in\Omega$. Therefore, $Q(Z_\Gamma[g])(\w)=(\g,\w)Z_\Gamma[\psi](\w)=Z_\Gamma[\Pi_g\psi](\w)=\calP Z_\Gamma[\Pi_g\psi](\w)$ for $m_{\wh\T}$-a.e $\w\in\Omega$. Then, since $Q$ and $\calP$ coincide on a dense set of $Z_\Gamma[S_\Gamma(\psi)]$ and both are continuos we must have $Q=\calP$. 
Thus, \eqref{eq:proj-hg} is proven. 

Finally, if $f\in S_\Gamma(\psi)$, $(i)$ follows from \eqref{eq:proj-hg}.

$(ii)$ This is a straightforward consequence of the characterization result \cite[Thm. 4.3]{BHP15} when identifying $\widehat{\Gamma}$ with $\Omega$.
\end{proof}

\begin{proof}[Proof of Theorem \ref{thm:ppal-invariant}]
Let $f\in S_\Gamma(\psi)$. Then, by Lemma \ref{lem:ppal-charcaterization}
we know that there exists a measurable function $\tilde{h}$ defined on $\Omega$ such that 
$Z_\Gamma[f](\w)=\tilde{h}(\w)Z_\Gamma[\psi](\w)$, for $m_{\wh\T}$-a.e. $\w\in\Omega$.

Now, take $h$ to be the $\Gamma^*$-periodization of $\tilde{h}$ defined on the whole $\wh \T$. That is, 
$$h(\w+\g^*):=\tilde{h}(\w)$$
for every $\g^*\in\Gamma^*$ and  $m_{\wh\T}$-a.e. $\w\in\Omega$. Then, by \eqref{eq:Zak-Upsilon} we have 
\begin{align*}
Z_\T[f](\w+\g^*)(x)&=
\sum_{j=0}^sJ_\sigma(-a_j,x)^{\frac1{2}}Z_\Gamma[f](\w)(\sigma_{-a_j}(x))(-a_j,\w+\g^*)\\
&=\sum_{j=0}^sJ_\sigma(-a_j,x)^{\frac1{2}}\tilde{h}(\w)Z_\Gamma[\psi](\w)(\sigma_{-a_j}(x))(-a_j,\w+\g^*)\\
&=\tilde{h}(\w)\sum_{j=0}^sJ_\sigma(-a_j,x)^{\frac1{2}}Z_\Gamma[\psi](\w)(\sigma_{-a_j}(x))(-a_j,\w+\g^*)\\
&=h(\w+\g^*)Z_\T[\psi](\w+\g^*)(x),
\end{align*}
for  $m_{\wh\T}$-a.e. $\w\in\Omega$ and $\mu$-a.e. $x\in C_\T$. This shows what we wanted.
\end{proof}

\section{Extra-invariance of $(\Gamma,\sigma)$-invariant spaces}\label{sec:extra-invariance}
In this section we shall characterize the  extra-invariance of $(\Gamma,\sigma)$-invariant spaces. To be precise,   let us first describe the setting we will work on.

Let $\Gamma\subseteq \Delta\subseteq \T$ be  discrete and countable LCA groups acting on a measure space $(\X,\mu)$, such that 
\begin{enumerate}
\item[$(a)$] $\Gamma$ is a uniform lattice of $\T$,
\item[$(b)$] the action of $\T$ coincides with that of $\Gamma$ (resp. $\Delta$) when restricted to $\Gamma\times \X$ (resp. $\Delta\times \X$).
\end{enumerate}

Our aim is to give necessary and sufficient conditions on a 
$(\Gamma,\sigma)$-invariant space $V\subseteq L^2(\X)$ to be also invariant under $\Delta$, that is, to satisfy $\Pi_\de f\in V$ for every $\de\in\Delta$ and for all $f\in V$. 

Towards this end, we define a partition of $\wh\T$ in $\Delta^*$-periodic measurable sets as follows: as before, let $\Omega\subseteq\wh\T$ be a measurable section for the quotient $\wh\T/\Gamma^*$ and  let $\calN\subseteq \Gamma^*$ be a section for the quotient $\Gamma^*/\Delta^*$
and for $\xi\in\calN$ define 
\begin{equation} \label{los Bs}
B_\xi:=\bigcup_{\delta^*\in\Delta^*}\Omega+\xi+\delta^*= \Omega+\xi+\Delta^*.
\end{equation}
By definition, for every $\xi\in\calN$, $B_\xi$ is measurable and $\Delta^*$-periodic,  meaning that $\B_\xi+\delta^*=B_\xi$ for all $\delta^*\in\Delta^*$.  Since $\Omega$ is a measurable section for $\wh\T/\Gamma^*$ and $\calN$ is a section for 
$\Gamma^*/\Delta^*$ if follows that $\{B_\xi\}_{\xi\in\calN}$ is a partition for 
$\wh\T$.

Now, given a $(\Gamma,\sigma)$-invariant space $V\subseteq L^2(\X)$, we define for every $\xi\in\calN$ the subspace
\begin{equation}\label{eq:subspace-U}
U_\xi:=\{f\in L^2(\X):\, Z_\T [f](\htau)={\bf 1 }_{B_\xi}(\htau)Z_\T [g](\htau), \textrm{ for }m_{\wh\T}\textrm{-a.e. }\htau\in\wh\T, \textrm{ with } g\in V\}.
\end{equation}
Our main result characterizes the $(\Delta, \sigma)$-invariance of $V$ in terms of the subspaces $U_\xi$, $\xi\in\calN$.

\begin{theorem}\label{thm:ppal}
Let $\Gamma\subseteq \Delta\subseteq \T$ be  discrete and countable LCA groups satisfying $(a)$ and $(b)$ and let $V\subseteq L^2(\X)$ be a $(\Gamma,\sigma)$-invariant space. Then, the following are equivalent:
\begin{enumerate}
\item [(i)] $V$ is $(\Delta,\sigma)$-invariant,
\item [(ii)] $U_\xi\subseteq V$ for all $\xi\in\calN$.
\end{enumerate}
Moreover, in case any of the above holds, we have that $V$ is the orthogonal
direct sum 
$$V=\bigoplus_{\xi\in\calN}U_\xi.$$
\end{theorem}
We postpone the proof of Theorem \ref{thm:ppal} to Subsection \ref{sec:proofs}. In the next subsection we give an heuristic that lies behind the statement of the above theorem. 

\subsection{Heuristic}\label{sec:heuristic}

In order to  better understand  the statement of Theorem \ref{thm:ppal},
we will give here some  heuristic  that will show how the Zak transform appears in the condition of the theorem. 
To that end, let us review some results for  invariant spaces under translations in LCA groups.
(See \cite{ACP10} for more details).

As in the previous section, let $\Gamma\subseteq \Delta\subseteq \T$ be  discrete and countable LCA groups satisfying $(a)$ and $(b)$.
Let $V\subseteq \ell^2(\T)$ be an invariant subspace by translations along $\Gamma$, i.e. such that 
$t_{\gamma} a \in V$ for every $a \in V \text{ and } \gamma \in \Gamma,$ where $t_\gamma$ is the usual translation operator given by $t_{\gamma}a(\tau)=a(\tau-\gamma)$.
If additionally, $V$ is invariant under  translations along the lattice $\Delta$ we say that $V$ has {\it extra invariance} $\Delta.$
A characterization of this extra invariance  was given in \cite{ACP10} and 
we  briefly  describe it here:
define $B_{\xi}$ as in \eqref{los Bs} and  let $\widetilde{U}_\xi$ be the subspaces of $\ell^2(\T)$ defined as  
\begin{equation}\label{eq:subspace-U-tilde}
\widetilde{U}_\xi:=\{a  \in \ell^2(\T):\, \wh{a} (\htau)={\bf 1 }_{B_\xi}(\htau)\wh{b}(\htau), \textrm{ for }m_{\wh\T}\textrm{-a.e. }\alpha\in\wh\T, \textrm{ with } b\in V\}.
\end{equation}
Here $\wh{a}$ denote the Fourier transform of $a$ for the group $\T$ which is given by 
$\wh{a}(\htau)=\sum_{\tau \in \T} a(\tau) \,(-\tau,\htau)$, for $\htau\in\wh\T$.
The result in  \cite{ACP10} adapted to our case can be stated, using the notation above, as:

\begin{theorem}\label{thm:CPA10}
Let $V$ be an invariant subspace of $\ell^2(\T)$ under translations along $\Gamma$. Then, the  following conditions  are equivalent:
\begin{enumerate}
\item [(i)]  $V$ has extra-invariance $\Delta$;
\item [(ii)] $\widetilde{U}_\xi\subseteq V$ for all $\xi\in\calN$.
\end{enumerate}
\end{theorem}

Essentially, Theorem \ref{thm:CPA10} says that a invariant subspace under translations in $\Gamma$ has extra-invariance if the Fourier transform of each function in $V$ can be cut off with the $B_{\xi}$ sets and still remains in the space $V$.

In trying to reproduce this result for subspaces invariant by the  action of $\Gamma$ on an abstract measure space $(\X,\mu)$
we found the problem that, because of the lack of the group structure, we can not define the Fourier transform on $\X$ .
However there are a way to surround this obstacle, as we will show next.

We start by decomposing  the space $\X$ in  orbits by the group action. For this, 
define the following relation in $\X$: for $x,y\in\X$  we say that $x$ is equivalent to 
$y$, $x \sim  y$, if there exist $\tau \in \T$ such that $\sigma_{\tau}(x)  = y$.
The properties of the action guarantee that $\sim$ is an equivalence relation.
If $C_\T \subseteq \X$ is the tilling set associated to $\sigma$, then $C_{\T}$ is a section of the quotient $\X/{\sim},$ the {\it orbit space}.

 Now  consider the following map,
$$ 
\Phi: L^2(\X) \longrightarrow  L^2(C_\T,\ell^2(\T)),  \qquad
f \longmapsto \Phi(f)(x)=\{J(\tau,x)^{\frac1{2}}f(\sigma_\tau(x))\}_{\tau \in \T}, \quad x \in C_\T.
$$
It is easily  seen that $\Phi$ is an isometric isomorphism (see e.g. the proof of \cite[Lemma 3.1]{BHP15}).

Let $V=S_{\Gamma}(\A)$ be a $(\Gamma,\sigma)$-invariant subspace of $L^2(\X)$ generated by a countable set $\A \subseteq L^2(\X)$. 
Define $\kJ,$ the {\it range function} associated to $V$, by
$\kJ:C_{\T}\longrightarrow \{\text{closed subspaces of } \ell^2(\T)\}$ 
$$ \kJ(x):=\overline{\mbox{span} }\{\Phi(\varphi)(x) :\;\varphi\in\A\} =
\overline{\mbox{span} }\{\{J(\tau,x)^{\frac1{2}}\varphi(\sigma_{\tau}(x))\}_{\tau \in \T}: \varphi\in\A\}.
$$

From the theory of range functions \cite{Hel64, Bow00} we conclude that $\kJ$ is a measurable range function and that
$$ f\in V  \Longrightarrow \Phi(f)(x) \in \kJ(x) \;\mu\text{ -a.e. }\,x\in C_{\T}.$$

On the other hand, the $(\Gamma,\sigma)$-invariance of $V$ forces $\kJ(x)$ to be invariant under translations along $\Gamma$ for almost all $x\in C_{\T}$. To see this, note that for $f\in L^2(\X)$ and $\gamma\in\Gamma$ we have that 
\begin{align*}
\Phi(\Pi_\gamma f)(x)&=\{J(\tau,x)^{\frac1{2}}\Pi_\gamma f(\sigma_\tau(x))\}_{\tau\in\T}\\
&=\{J(\tau,x)^{\frac1{2}}J(-\gamma,\sigma_\tau(x))^{\frac1{2}}f(\sigma_{\tau-\gamma}(x))\}_{\tau\in\T}\\
&=\{J(\tau-\gamma,x)^{\frac1{2}}f(\sigma_{\tau-\gamma}(x)\}_{\tau\in\T}=t_\gamma(\Phi(f)(x)),
\end{align*}
where we have used the properties of the Jacobian (see e.g. \cite[Subsection 2.2]{ BHP15}). 
Moreover, when $V$ is, additionally,  a $(\Delta,\sigma)$-invariant space, $\kJ(x)$ has extra invariance $\Delta$  for almost all $x\in C_{\T}$. 
The good news  is that the subspaces $\kJ(x)$ are subspaces of $\ell^2(\T)$ i.e. subspaces of functions defined 
in a group, so we can use the Fourier transform and apply Theorem \ref{thm:CPA10}.  With this purpose, for almost $\mu$-a.e.  $x \in C_{\T}$,  we define, associated to each $\kJ(x)$, the  subspaces defined by  \eqref{eq:subspace-U-tilde} that read
\begin{equation*}
\widetilde{U}_\xi(x):=\{a\in \ell^2(\T):\, \wh{a}(\htau)={\bf 1 }_{B_\xi}(\htau)\wh{b}(\htau), \textrm{ for }m_{\wh\T}\textrm{-a.e. }\htau\in\wh\T, \textrm{ with } b\in \kJ(x)\}.
\end{equation*}
Then,  Theorem \ref{thm:CPA10}  applied to $\kJ(x)$ can be stated, using the notation above, as:

\begin{theorem}\label{thm:ppal2}
For  $\mu$-a.e. $x \in C_{\T}$  the following conditions are equivalent:
\begin{enumerate}
\item [(i)] The invariant space $\kJ(x)$ along translations in $\Gamma$ is also invariant under translations in $\Delta$;
\item [(ii)] $\widetilde{U}_\xi(x)\subseteq \kJ(x)$ for all $\xi\in\calN$.
\end{enumerate}
\end{theorem}

Now given $f \in L^2(\X)$ and $x \in C_{\T}$ we note that
$$\wh{\Phi(f)(x)}(\htau)=\sum_{\tau \in \T} \Phi(f)(x)(\tau) \,(-\tau,\htau) = Z_\T [f](-\htau)(x)$$
which motives the definition of the cut off sets given in \eqref {eq:subspace-U} and the statement of Theorem \ref{thm:ppal}.

\subsection{Proof of Theorem \ref{thm:ppal}}\label{sec:proofs}
The proof of Theorem \ref{thm:ppal} will not follow the ideas of the previous section. Instead, we will follow that lines of the proofs of \cite[Theorem 5.1]{ACP11}. See also \cite{ACP10, ACHKM10}.
We need a fundamental intermediate lemma. 

\begin{lemma}\label{lem:u-is-invariant}
Let $V\subseteq L^2(\X)$ and $\xi\in\calN$. Assume that the subspaces $U_\xi$ defined in \eqref{eq:subspace-U} satisfied $U_\xi\subseteq V$. Then, $U_\xi$ is $(\Gamma,\sigma)$-invariant.  
\end{lemma} 

\begin{proof}
We split the proof into three parts.

\underline{Step 1:} Let us show that $U_\xi$ is closed. Take $\{f_n\}_{n\in\N}\subseteq U_\xi$ be a sequence of functions converging to $f\in L^2(\X)$. Since $U_\xi\subseteq V$ and $V$ is closed, we have that $f\in V$. Note that, since for any $g\in U_\xi$, we have that  $Z_\T [g](\htau)={\bf 1 }_{B_\xi}(\htau)Z_\T [g](\htau), \textrm{ for }m_{\wh\T}\textrm{-a.e. }\htau\in\wh\T$ we obtain for $n\in\N$, 
\begin{align*}
\|f-f_n\|^2&=\|Z_\T[f]-Z_\T[f_n]\|^2_{L^2(\wh\T,L^2(C_\T))}
=\int_{\wh\T}\|Z_\T[f](\htau)-Z_\T[f_n](\htau)\|^2_{L^2(C_\T)}\,dm_{\wh\T}(\htau)\\
&=\int_{\wh\T}\|Z_\T[f](\htau)-{\bf 1}_{B_\xi}(\htau)Z_\T[f_n](\htau)\|^2_{L^2(C_\T)}\,dm_{\wh\T}(\htau)\\
&=\int_{B_\xi}\|Z_\T[f](\htau)-Z_\T[f_n](\htau)\|^2_{L^2(C_\T)}\,dm_{\wh\T}(\htau)
+\int_{B_\xi^c}\|Z_\T[f](\htau)\|^2_{L^2(C_\T)}\,dm_{\wh\T}(\htau).
\end{align*}
Since $f_n\to f$ in $L^2(\X)$, we conclude from the above equality that 
$Z_\T[f](\htau)=0$ for $m_{\wh\T}$-a.e. $\htau\in B_\xi^c$ and then, 
we must have $Z_\T [f](\htau)={\bf 1 }_{B_\xi}(\htau)Z_\T [f](\htau)$
for $m_{\wh\T}$-a.e. $\htau\in \wh\T$. Therefore, using that $f\in V$, it follows that $f\in U_\xi$ proving that $U_\xi$ is closed.

\underline{Step 2:} Note that $U_\xi$ is $(\Gamma,\sigma)$-invariant. Indeed, 
for $f\in U_\xi$ and $\g\in\Gamma$ we have that 
$$Z_\T [\Pi_\g f](\htau)=(\g,\htau)Z_\T [f](\htau)={\bf 1 }_{B_\xi}(\htau)(\g,\htau)Z_\T [g](\htau)={\bf 1 }_{B_\xi}(\htau)Z_\T [\Pi_\g g](\htau)$$ 
 where $g\in V$ is such that $Z_\T [f](\htau)={\bf 1 }_{B_\xi}(\htau)Z_\T [g](\htau)$ for $m_{\wh\T}$-a.e. $\htau\in \wh\T$. Since $V$ is $(\Gamma,\sigma)$-invariant, $\Pi_\g g\in V$ and by construction of $U_\xi$, $\Pi_\g f\in U_\xi$.

\underline{Step 3:}
Given $f\in U_\xi$ and $\delta_0\in\Delta$ we shall prove now that $\Pi_{\delta_0} f\in U_\xi$. This part is not as direct as the $(\Gamma,\sigma)$-invariance of $U_\xi$ and it requires more work. 

 Let $g\in V$ be such that $Z_\T [f](\htau)={\bf 1 }_{B_\xi}(\htau)Z_\T [g](\htau)$
for $m_{\wh\T}$-a.e. $\htau\in \wh\T$. Then, 
$$Z_\T [\Pi_{\delta_0} f](\htau)=(\delta_0,\htau)Z_\T [f](\htau)={\bf 1 }_{B_\xi}(\htau)(\delta_0,\htau)Z_\T [g](\htau)$$ 
for $m_{\wh\T}$-a.e. $\htau\in \wh\T$. Although we have that $(\delta_0,\htau)Z_\T [g](\htau)=Z_\T [\Pi_{\delta_0}g](\htau)$, we can not say by now that $\Pi_{\delta_0}g$ belongs to $V$ to deduce that $\Pi_{\delta_0}f\in U_\xi$. However, suppose that we can find  a $\Gamma^*$-periodic function $h$ defined on $\wh\T$ such that
$$h(\htau)=(\delta_0,\htau)\,\,\textrm{ for }m_{\wh\T}\textrm {-a.e. } \htau\in B_\xi.$$
Then, we would have  that
$$Z_\T [\Pi_{\delta_0} f](\htau)={\bf 1 }_{B_\xi}(\htau)h(\htau)Z_\T [g](\htau)$$ 
for $m_{\wh\T}$-a.e. $\htau\in \wh\T$ and, by Theorem \ref{thm:ppal-invariant}, $\Pi_{\delta_0} f\in S_\Gamma(g)\subseteq U_\xi$ because, as we already saw, $U_\xi$ is $(\Gamma,\sigma)$-invariant. 

Therefore, it is sufficient find such a $\Gamma^*$-periodic function $h$ and this is what we do next. 
For every $\delta^*\in\Delta^*$ we know that 
$$(\delta_0,\w+\xi)=(\delta_0,\w+\xi +\delta^*)\,\,\textrm{ for }m_{\wh\T}\textrm {-a.e. } \w\in\Omega.$$
We define 
$$h(\w+\g^*):=(\delta_0,\w+\xi)\,\,\textrm{ for }m_{\wh\T}\textrm {-a.e. } \w\in\Omega\,\,
\textrm{ and }\forall\,\,\g^*\in\Gamma^*.$$
By definition, $h$ is $\Gamma^*$-periodic. To see that  $h=(\delta_0, \cdot)$ on $\B_\xi$ take $\htau=\w+\xi+\delta^*\in B_\xi$. Then, 
$$h(\htau)=h(\w)=(\delta_0, \w+\xi)=(\delta_0,\w+\xi +\delta^*)=(\delta_0,\htau),$$
as we wanted. This conclude the proof.
\end{proof}
 
 \begin{proof}[Proof of Theorem \ref{thm:ppal}]
 $(i)\Rightarrow(ii)$.
Let $\xi\in\calN$ be fixed and take $f\in U_\xi$. Then, there exists $g\in V$ such that  
$Z_\T [f](\htau)={\bf 1 }_{B_\xi}(\htau)Z_\T [g](\htau)$
for $m_{\wh\T}$-a.e. $\htau\in \wh\T$. Since ${\bf 1 }_{B_\xi}$ is a $\Delta^*$-periodic function,  Theorem \ref{thm:ppal-invariant} but applied to $\Delta$ instead of $\Gamma$ gives that $f\in S_\Delta(g)$. Using that $V$ is $(\Delta, \sigma)$-invariant, we obtain $S_\Delta(g)\subseteq V$. Thus $(ii)$ is proven.

 $(ii)\Rightarrow(i)$.
 Since $\{B_\xi\}_{\xi\in\calN}$ is a partition of $\wh \T$,  the subspaces $\{U_\xi\}_{\xi\in\calN}$ are mutually orthogonal. 
 Then $(ii)$ gives that $\bigoplus_{\xi\in\calN}U_\xi\subseteq V$.
 
On the other hand, for every $\xi\in\calN$ call  $P_{U_\xi}$ the orthogonal projection onto $U_\xi$. Thus, using again that  $\{U_\xi\}_{\xi\in\calN}$ are mutually orthogonal, we have that for every $f\in V$, $f=\sum_{\xi\in\calN} P_{U_\xi}f$. Therefore, $V\subseteq\bigoplus_{\xi\in\calN}U_\xi$ and then, 
 $$V=\bigoplus_{\xi\in\calN}U_\xi.$$
Finally, since by  Lemma \ref{lem:u-is-invariant}, $U_\xi$ is a $(\Delta, \sigma)$-invariant space for every $\xi\in\calN$, so is $V$.
 \end{proof}
 
\begin{remark}
Observe that given any $\Gamma\subseteq \Delta\subseteq \T$ satisfying $(a)$ and $(b)$ there is always a $(\Gamma, \sigma)$-invariant space of $L^2(\X)$ that is also $(\Delta, \sigma)$-invariant.
 For seeing this, assume that $e\in\wh \T$, the neutral element of $\wh \T$, belongs to $\calN$ and take  $\varphi\in L^2(\X)$ such that $Z_\T[\varphi]=\chi_{B_e}$. 
 Set $V=S_\Gamma(\varphi)$ and let us prove that $V$ is 
 $(\Delta, \sigma)$-invariant. Note that, by Theorem \ref{thm:ppal-invariant}, for every  $g\in V$ there there exists a measurable $\Gamma^*$-periodic function $h$ defined on $\wh\T$ such that 
$Z_\T[f]=hZ_\T[\varphi].$
Therefore, for $\xi\in\calN$ 
\begin{align*}U_\xi=&\{f\in L^2(\X):\, Z_\T [f]={\bf 1 }_{B_\xi}h Z_\T [\varphi],\textrm{ with } h  \textrm{ a } \Gamma^*-\textrm{ periodic function} \}\\
=&\{f\in L^2(\X):\, Z_\T [f]={\bf 1 }_{B_\xi}h{\bf 1 }_{B_e},\textrm{ with } h  \textrm{ a } \Gamma^*-\textrm{ periodic function} \}
\end{align*}
and it is easily seen that $U_\xi=\{0\}$ when $\xi\neq e$ and $U_e=V$. Then, by Theorem 
\ref{thm:ppal}, we conclude that $V$ is $(\Delta, \sigma)$-invariant. 
\end{remark}

\begin{example}\noindent

{\bf $\bullet$ $\Z$ acts in $\R^2$ by multiplication with the shear matrix.}

Let us consider the group $\Gamma=\Z$ acting on $\X=\R^2$ by
  $$\sigma(k, (x,y))=\begin{pmatrix}
                        1&0\\
                        k&1
                    \end{pmatrix}\begin{pmatrix}
                      x\\
                      y
                    \end{pmatrix}=\begin{pmatrix}
                      x\\
                      kx+y
                    \end{pmatrix}.$$ 
Take $\T=\frac1{6}\Z$ acting on $\X=\R^2$ by
$$\sigma^\T(\frac1{6}j, (x,y))=\begin{pmatrix}
                        1&0\\
                        \frac1{6}j&1
                    \end{pmatrix}\begin{pmatrix}
                      x\\
                      y
                    \end{pmatrix}=\begin{pmatrix}
                      x\\
                      \frac1{6}jx+y
                    \end{pmatrix},$$ 
which is an action that clearly extends $\sigma$. Then, $J(k,(x,y))=1$ for all $k\in\Z$ and $(x,y)\in\R^2$ and also 
$J_\T(\frac1{6}j,(x,y))=1$ for all $j\in\Z$ and $(x,y)\in\R^2$. Moreover, $\sigma_\T$ has the tiling property with tiling  set $C_\T=\{(x,y)\in\R^2:\, 0\leq y\leq \frac1{6}x\}\cup\{(x,y)\in\R^2:\, \frac1{6}x\leq y\leq 0\}$. Since $\F=\{0, -\frac1{6},\cdots,-\frac{5}{6}\}$ is a section for the quotient $\T/\Gamma$, by  \eqref{eq:tiling-set} we then have that $ C_\Gamma=\bigcup_{j=0}^5\sigma^\T_{-\frac1{6}j}(C_\T)=\{(x,y)\in\R^2:\, 0\leq y\leq x\}\cup\{(x,y)\in\R^2:\, x\leq y\leq 0\}$. In Figure 1 the tiling set $C_\T$ is  represented in light gray. According the gray of the regions is getting darker, we have illustrated the sets $\sigma^\T_{-\frac1{6}j}(C_\T)$
for $1\leq j\leq5$. 

If we now consider the group $\Delta=\frac1{2}\Z$ and we take $V\subseteq L^2(\R^2)$ a $(\Gamma, \sigma)$-invariant space that is also $(\Delta, \sigma)$-invariant, then the subspaces $U_0$ and $U_1$ given by \eqref{eq:subspace-U} are 
$$U_i:=\{f\in L^2(\R^2):\, Z_\T [f]={\bf 1 }_{B_i}Z_\T [g],\textrm{ with } g\in V\}, \quad\textrm{for }\quad i=0,1,$$
where $B_0=[0,1)\cup[2,3)\cup[4,5)$ and $B_1=[1,2)\cup[3,4)\cup[5,6)$.

\begin{figure}[h]
\includegraphics[scale=0.5]{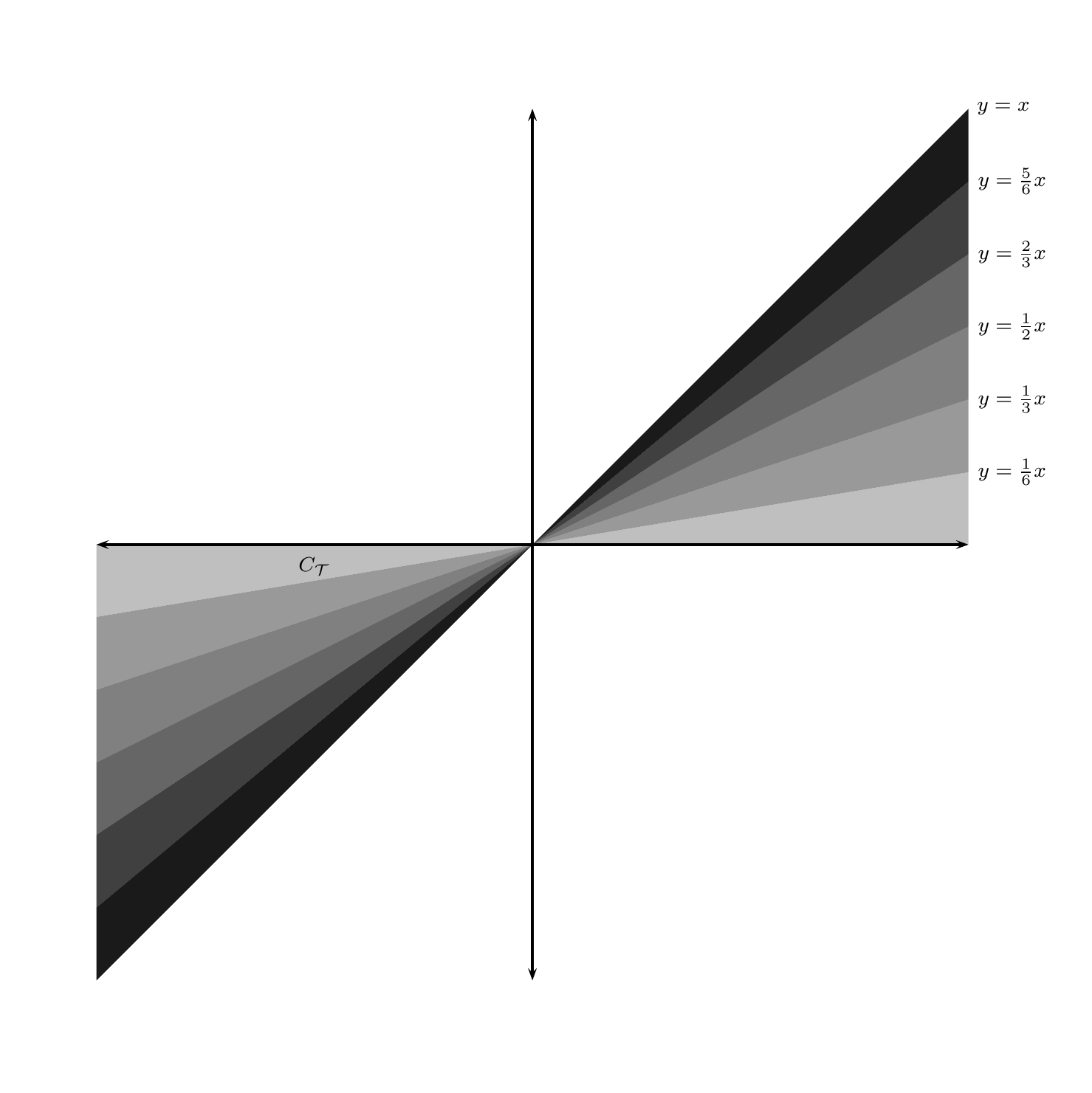}
\centering 
\vspace*{-0.5cm}
\caption{\small Tiling set for the action by the shear matrix.}
\end{figure}

{\bf $\bullet$ $\Z$ acts in $\R^2$ by dilations.}
Suppose now that  the group $\Gamma=\Z$ acts on $\X=\R^2$ by
  $$\sigma(k, (x,y))=\begin{pmatrix}
                        2^k&0\\
                        0&2^{\frac{k}{2}}
                    \end{pmatrix}\begin{pmatrix}
                      x\\
                      y
                    \end{pmatrix}=\begin{pmatrix}
                      2^kx\\
                      2^{\frac{k}{2}}y
                    \end{pmatrix}.$$ 
Take $\T=\frac1{4}\Z$ acting on $\X=\R^2$ as a natural  extension of  $\sigma$. That is, 
$$\sigma^\T(\frac1{4}j, (x,y))=\begin{pmatrix}
                        2^{\frac{j}{4}}&0\\
                        0&2^{\frac{j}{8}}

                    \end{pmatrix}\begin{pmatrix}
                      x\\
                      y
                    \end{pmatrix}=\begin{pmatrix}
                      2^{\frac{j}{4}}x\\
                      2^{\frac{j}{8}}y
                    \end{pmatrix},$$ 
The Jacobians in this case are $J(k,(x,y))=2^{\frac{3k}{2}}$ and $J_\T(\frac1{4}j,(x,y))=2^{\frac{3k}{8}}$ for all $k, j\in\Z$ and $(x,y)\in\R^2$.
The tiling set associated to $\sigma_\T$ is 
$$C_\T=[-2^{-\frac{3}{4}}, 2^{-\frac{3}{4}}]\times [-2^{-\frac{3}{8}}, 2^{-\frac{3}{8}}]\setminus[-\frac1{2}, \frac1{2}]\times[-2^{-\frac{1}{2}},2^{-\frac{1}{2}}].$$
Note that since $\F=\{0, -\frac1{4},-\frac1{2},-\frac{3}{4}\}$, by \eqref{eq:tiling-set},
we that have that the tiling set associated to $\sigma$ is 
$C_\Gamma=\sigma^\T_{0}(C_\T)\cup\sigma^\T_{-\frac1{4}}(C_\T)\cup\sigma^\T_{-\frac1{2}}(C_\T)\cup\sigma^\T_{-\frac{3}{4}}(C_\T)$ which coincides  with$[-1,1]^2\setminus[-\frac1{2}, \frac1{2}]\times[-2^{-\frac{1}{2}},2^{-\frac{1}{2}}]$. This is illustrated in Figure 2 where $x_1=2^{-\frac{3}{4}}$, $x_2=2^{-\frac{1}{2}}$,  $x_3=2^{-\frac{1}{4}}$ and $y_1=2^{-\frac{1}{2}}$, $y_2=2^{-\frac{3}{8}}$, $y_3=2^{-\frac{1}{4}}$, $y_4=2^{-\frac{1}{8}}$. The tiling set $C_\T$ is represented by the frame painted in  lightgray and the other gray regions represent the dilations of $C_\T$ by $\frac1{4}$, $\frac1{2}$ and $\frac{3}{4}$ respectively. 

Now, for  $\Delta=\frac1{2}\Z$, we have that the  partition $\{B_\xi\}_{\xi=0,1}$ of $\widehat{\T}=[ 0,4)$ is given by 
$B_0=[0,1)\cup[2,3)$ and $[1,2)\cup[3,4).$
Thus, if  $V\subseteq L^2(\R^2)$ is a $(\Gamma, \sigma)$-invariant space that is also $(\Delta, \sigma)$-invariant, $U_i:=\{f\in L^2(\R^2):\, Z_\T [f]={\bf 1 }_{B_i}Z_\T [g],\textrm{ with } g\in V\}$ for $i=0,1$.
\end{example}

\begin{figure}[h]
\includegraphics[scale=0.7]{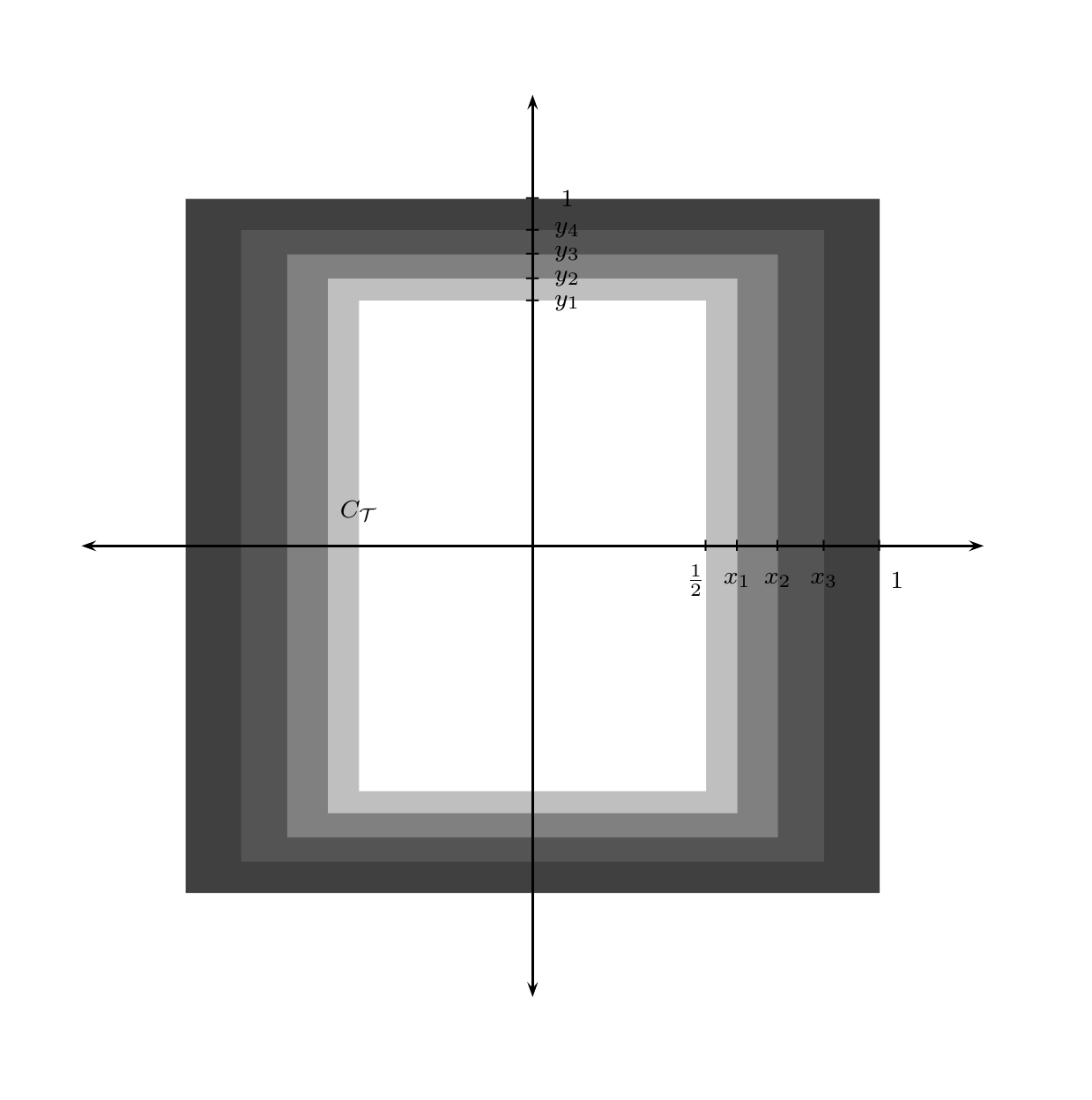}
\centering 
\vspace*{0.3cm}
\caption{\small Tiling set for the action by dilations.}
\end{figure}

\section{$(\Gamma, \sigma)$-invariant spaces and decomposable MI spaces.}\label{sec:decomposable}
 
In this section we want to connect $(\Gamma, \sigma)$-invariant spaces which are also $(\Delta, \sigma)$-invariant with decomposable multiplicatively invariant spaces. For this, we start be recalling the basics of multiplicatively invariant spaces and their connection with invariant spaces. Then we will prove a characterization of extra invariance of $(\Gamma, \sigma)$-invariant spaces in terms of decomposable multiplicatively invariant spaces. Additionally we will solve an approximation problem of finding the nearest $(\Gamma, \sigma)$-invariant space to a given set of data. 

\subsection{Characterization of extra invariance in terms of decomposable MI spaces}

Recall that given a separable Hilbert space $\calH$ and a measure space $(\mathcal{Y},\nu)$ a multiplicatively invariant  space $M$ (MI-space for short)  is a  closed subspace of  $L^2(\mathcal{Y}, \mathcal{H})$  which is invariant under multiplications of functions in a determining set $D\subseteq L^{\infty}(\mathcal{Y})$, where $D$ is a {\it determining set for $L^1(\mathcal{Y})$} if for every $f\in L^1(\mathcal{Y})$ such that $\int_{\mathcal{Y}} fg\, d\nu=0$ for all $g\in D,$ one has $f=0.$
For details in the subject see \cite[Section 2]{BR14} as well as \cite{Hel64, Sri64, HS64}.

For the notion of decomposable MI-spaces, introduced in \cite{CMP17}, fix  $D$  a determining set for $L^1(\mathcal{Y})$ and for a given $\kappa\in\N\cup\{+\infty\}$,  suppose that a Hilbert space $\calH$ can be decomposed into an  orthogonal sum as 
\begin{equation}\label{eq:descomposicion}
\mathcal{H}=\mathcal{H}_1\oplus\cdots\oplus \mathcal{H}_{\kappa}.
\end{equation}
For each MI-subspace $M,$ we define 
\[
M_j:= \{\calP_j(\Phi)\colon \Phi\in M\}, \mbox{ for all } 1\le j\le \kappa,
\]
where $\calP_j(\Phi)$ denotes the function in $L^2(\mathcal{Y}, \mathcal{H})$ defined as $\calP_j(\Phi)(y):=P_{\calH_j}(\Phi(y))$  with  $P_{\mathcal{H}_j}\colon \mathcal{H}\to\mathcal{H}_j$ being  the orthogonal projection of $\calH$ onto $\calH_j$. Note that $\calP_j$ is the orthogonal projection of $L^2(\mathcal{Y}, \mathcal{H})$ into $L^2(\mathcal{Y}, \mathcal{H}_j)$, for all $j=1,\dots,\kappa$.

\begin{definition}\label{def:decomposable-MI}
 We say that a MI-subspace $M\subseteq L^2(\mathcal{Y}, \mathcal{H})$ is {\it decomposable} with respect to  $\{\calH_1, \ldots, \calH_\kappa\}$ if 
\[
M_j\subseteq M \quad \mbox{ for all } 1\le j\le\kappa. 
\]
\end{definition}

To establish the connection between decomposable MI-spaces and  $(\Gamma, \sigma)$-invariant spaces that are also $(\Delta, \sigma)$-invariant suppose that 
 $\Gamma\subseteq \Delta\subseteq \T$ are  discrete and countable LCA groups acting on a measure space $(\X,\mu)$ such that they satisfy properties $(a)$ and $(b)$ of Section  \ref{sec:extra-invariance} and keep the notation stated in that section.

Observe that,  by Proposition \ref{prop:Zak}, $(\Gamma, \sigma)$-invariant spaces in $L^2(\X)$ are in correspondence with MI-spaces in $L^2(\W, L^2(C_\Gamma))$ under the Zak transform $Z_\Gamma$ where the underlying determinig set $D\subseteq L^\infty(\W)$ is $D=\{e_\g\}_{\g\in\Gamma}$, with $e_\gamma(\w):=(\gamma,\w)$ for $\w\in \W$. 
In what follows we need to consider a slightly different version of $Z_\Gamma$ which can be seen as a Fourier transform of it. More precisely, let $Z:L^2(\X) \longrightarrow L^2(\W, \left(L^2(C_\T)\right)^{s+1})$
be defined by
$$
Z[f](\w)(x):=(Z_\T[f](\w+\g_0^*)(x),\dots,Z_\T[f](\w+\g_s^*)(x)),$$
where $\w\in\W$ and $x\in C_\T$. 
Note that from Remark \ref{rem:Zak-matrix} and Proposition \ref{prop:Zak} where we identify $\widehat{\Gamma}$ with $\W$, we obtain that 

\begin{equation}\label{eq:zak-nueva}
Z[\Pi_\g f]=e_\g Z[f]\quad\textrm{ for all }\quad f\in L^2(\X), \g\in\Gamma.
\end{equation}
 Moreover, with the same notation of Remark \ref{rem:Zak-matrix} and noticing that $F$ is a unitary matrix we have 
\begin{align*}
\|Z[f]\|^2&=\int_\W \|Z[f](w)\|^2\,dm_{\wh\T}(\w)
=\int_\W \|F.D(\w,\cdot).V(\w,\cdot))\|^2\,dm_{\wh\T}(\w)\\
&=\int_\W \|D(\w,\cdot).V(\w,\cdot))\|^2\,dm_{\wh\T}(\w)
=\int_\W \sum_{j=0}^s\|\left[D(\w,\cdot).V(\w,\cdot)\right]_j\|_{L^2(C_\T)}^2\,dm_{\wh\T}(\w)\\
&=\int_\W \sum_{j=0}^s\int_{C_\T}|J_\sigma(-a_j,x)^{\frac1{2}}Z_\Gamma[f](\w)(\sigma_{-a_j}(x))|^2\,d\mu(x)\,dm_{\wh\T}(\w)\\
&=\int_\W \int_{C_\Gamma}|Z_\Gamma[f](\w)(y)|^2\,d\mu(y)\,dm_{\wh\T}(\w)=\|f\|^2_{L^2(\X)}.
\end{align*}
Here, in the last equality we have used \eqref{eq:tiling-set} and the definition of $J$.
This shows that $Z$ is an isometry and we can conclude that it is an isometric isomorphism (i.e. surjective) because so is $Z_\Gamma$. 

Furthermore,  by \eqref{eq:zak-nueva}, it follows that $V\subseteq L^2(\X)$ is a $(\sigma,\Gamma)$-invariant space if and only if $Z[V]$ is an MI-space in  $L^2(\W, \left(L^2(C_\T)\right)^{s+1})$ with respect to the determining set  $D=\{e_\g\}_{\g\in\Gamma}$.

Now we define a decomposition  of $\mathcal{H}= \left(L^2(C_\T)\right)^{s+1}$ in the following way: for each $\xi\in\mathcal{N}$ let
\[
\mathcal{H}_{\xi}:= \{F\in \left(L^2(C_\T)\right)^{s+1}\colon F_i=0 \mbox{ for all }i \mbox{ such that }\g_i^*\notin\xi+\Delta^* \}.
\]
Then we clearly have that  $\left(L^2(C_\T)\right)^{s+1}=\bigoplus_{\xi\in\mathcal{N}}\mathcal{H}_{\xi}$.
The characterization of $(\Gamma, \sigma)$-invariant spaces which are also $(\Delta, \sigma)$-invariant in terms of  decomposable MI-spaces is stated in the following theorem. 

\begin{theorem}\label{thm:decomposable}
Let $\Gamma\subseteq \Delta\subseteq \T$ be  discrete and countable LCA groups satisfying $(a)$ and $(b)$ of Section  \ref{sec:extra-invariance} and let $V\subseteq L^2(\X)$ be a $(\Gamma,\sigma)$-invariant space. Then, the following are equivalent:
\begin{enumerate}
\item [(i)] $V$ is $(\Delta,\sigma)$-invariant,
\item [(ii)] $Z[V]\subseteq L^2(\W, \left(L^2(C_\T)\right)^{s+1})$ is decomposable MI-space with respect to $\{\mathcal{H}_{\xi}\}_{\xi\in\mathcal{N}}.$
\end{enumerate}
\end{theorem} 

\begin{proof}
By Theorem \ref{thm:ppal} we know that $V$ is $(\Delta,\sigma)$-invariant if and only if 
$U_\xi\subseteq V$ for all $\xi\in\mathcal{N}$, where $U_\xi$ is as in \eqref{eq:subspace-U}. Since $U_\xi\subseteq V$ if and only if  $Z[U_\xi]\subseteq Z[V]$, it is enough to show that  $Z[U_\xi]=\calP_\xi Z[V]$ for all $\xi\in\mathcal{N}$ where $\calP_\xi(\Phi)(\w)=P_{\calH_\xi}(\Phi(\w))$ with $P_{\calH_\xi}:
L^2(\W, \left(L^2(C_\T)\right)^{s+1}) \longrightarrow L^2(\W, \mathcal{H}_{\xi})$ being the 
orthogonal projection onto $L^2(\W, \mathcal{H}_{\xi}).$

Fix $\xi\in\mathcal{N}$ and $g\in V$. For $f\in L^2(\X)$ such that  $Z_\T[f]={\bf 1 }_{B_\xi}Z_\T[g]$, we have that 
$$Z[f](\w)(x)=({\bf 1 }_{B_\xi}(\w+\g_0^*)Z_\T[g](\w+\g_0^*)(x),\dots,{\bf 1 }_{B_\xi}(\w+\g_s^*)Z_\T[g](\w+\g_s^*)(x))$$
for $m_{\widehat{\T}}$- a.e. $\w\in\W$ and $\mu$- a.e. $x\in C_\T$. Then, since for all $j\in\{0,\dots, s\}$,
$$\left[Z[f](\w)(x)\right]_j=
\begin{cases}
0&\textrm{ if }j\textrm{ is such that  }\g_j^*\notin\xi+\Delta^* \\
Z_\T[g](\w+\g_j^*)(x)&\textrm{ otherwise},
\end{cases}
$$ 
it follows that $Z[f]=\calP_\xi Z[g]$.  Therefore,
$\calP_\xi Z[V]=Z[U_\xi]$, and the result follows.
\end{proof}

\subsection{Approximation problem}

In \cite{ACHM07} it was solved the problem of finding a shift-invariant space (i.e. a closed subspace of $L^2(\R^d)$ invariant under integer translations)  that best fits a given set of data in the sense of least squares. Later, in \cite{CM15} the same problem was solved but for shift-invariant spaces having some extra invariance. The question that arises here is if one can state and solve this approximation problem in the context of $(\Gamma, \sigma)$-invariant spaces. For this, given a finitely generated  $(\Gamma, \sigma)$-invariant space $V$, define the {\it length of $V$}, $\ell(V)$,  as the minimum number of function we need to generate it  and consider the classes 
$$\calC_\ell:=\{V\subseteq L^2(\X): V \textrm{ is a finitely generated } (\Gamma, \sigma)\textrm{-invariant space with }\ell(V)\leq \ell\},$$  
with $\ell$ being a fixed number in $\N$ and 
$$\calC^\Delta_\ell=\{V\in\calC_\ell: V \textrm{ is also }(\Delta, \sigma)\textrm{-invariant}\}.$$
Then, the approximation problems can be stated as follows: 
\begin{itemize}
\item [(1)]Given $\Psi=\{\psi_1,\dots,\psi_m\}\subseteq L^2(\W,L^2(C_{\Gamma}))$, find $V^*\in\calC_\ell$ such that 
$$\sum_{j=1}^m \|\psi_j- P_{V^*} \psi_j\|^2 \le \sum_{j=1}^m \|\psi_j- P_{V} \psi_j\|^2,$$
for all $V\in\calC_\ell$, where $P_V$ and $P_{V^*}$ denotes the orthogonal projections onto $V$ and $V^*$ respectively.  
\item [(2)]Given $\Psi=\{\psi_1,\dots,\psi_m\}\subseteq L^2(\W, \left(L^2(C_\T)\right)^{s+1})$, find $W^*\in\calC^\Delta_\ell$ such that 
$$\sum_{j=1}^m \|\psi_j- P_{W^*} \psi_j\|^2 \le \sum_{j=1}^m \|\psi_j- P_{W} \psi_j\|^2,$$
for all $W\in\calC^\Delta_\ell$.
\end{itemize}

The solutions for problems (1) and (2) are obtained as combinations of severals results: on one hand, we have that 
  $(\Gamma, \sigma)$-invariant spaces are in one-to-one correspondence with  MI-spaces and $(\Gamma, \sigma)$-invariant spaces with extra invariance in $\Delta$ are characterized in terms of decomposable MI-spaces in Theorem \ref{thm:decomposable}. On the other hand, Problems (1) and (2) have their version for MI-spaces and were solved in \cite[Theorems 4.2 and 4.6]{CMP17}.
Then, as a consequence we have:
\begin{theorem}
Let $\ell\in \N$ and $\Psi=\{\psi_1,\dots,\psi_m\}\subseteq L^2(\W,L^2(C_{\Gamma}))$. Then, there exists a $(\Gamma, \sigma)$-invariant space $W^*$ of length $\ell$ and having extra invariance in $\Delta$  solving Problem (2), that is, satisfying
\begin{equation*}
\sum_{j=1}^m \|\psi_j- P_{V^*} \psi_j\|^2 \le \sum_{j=1}^m \|\psi_j- P_{V} \psi_j\|^2,
\end{equation*}
for all $V\in\calC_\ell$. 
\end{theorem}

\begin{theorem}
Let $\ell\in \N$ and $\Psi=\{\psi_1,\dots,\psi_m\}\subseteq L^2(\W,\left(L^2(C_\T)\right)^{s+1})$. Then, there exists a $(\Gamma, \sigma)$-invariant space $W^*$ of length $\ell$ solving Problem (2), that is, satisfying
\begin{equation*}
\sum_{j=1}^m \|\psi_j- P_{W^*} \psi_j\|^2 \le \sum_{j=1}^m \|\psi_j- P_{W} \psi_j\|^2,
\end{equation*}
for all $W\in\calC^\Delta_\ell$. 
\end{theorem}

The proofs of the above theorems are easy  adaptations of that of \cite[Theorem 5.1 and 5.5]{CMP17} so, we left the details for the reader. 

\end{document}